\documentclass[12pt]{amsart}
\usepackage{cite}
\usepackage{graphicx}
\usepackage{amsmath, amsfonts, amssymb,latexsym, graphics, graphicx, mathrsfs, manfnt, textcomp, appendix, chemarrow}
\usepackage{hyperref}
\usepackage[all]{xy}

\author{Elif Uyan\i k}
\address{Elif Uyan\i k\\
Department of Mathematics\\
Middle East Technical University\\
06800 Ankara\\
Turkey}
\email{euyanik@metu.edu.tr}

\author{Murat H. Yurdakul}
\address{Murat H. Yurdakul\\
Department of Mathematics\\
Middle East Technical University\\
06800 Ankara\\
Turkey}
\email{myur@metu.edu.tr}

\title[]{Bounded Operators to $\ell$-K\"{o}the Spaces}

\thanks{This research was partially supported by Turkish Scientific and Technological Research Council.}

\numberwithin{equation}{section}

\theoremstyle{thmit} % Numbered and Italic
\newtheorem{theorem}{Theorem}[section]

\newtheorem{lemma}[theorem]{Lemma}

\begin{document}

\maketitle

\begin{abstract}
    For Fr{\'e}chet spaces $E$ and $F$ we write $(E,F)\in \mathcal{B}$ if every continuous linear operator from $E$ to $F$ is bounded. Let $\ell$ be a Banach sequence space with a monotone norm in which the canonical system $(e_{n})$ is an unconditional basis. We obtain a necessary and sufficient condition for $(E,F) \in \mathcal{B}$ when $F = \lambda^{\ell}(B).$ We say that a triple $(E,F,G)$ has the bounded factorization property and write $(E,F,G) \in \mathcal{BF}$ if each continuous linear operator $T : E \longrightarrow G$ that factors over $F$ is bounded. We extend some results in \cite{Ter03} to $\ell$-K\"{o}the spaces and obtain a sufficient condition for $(E,\lambda^{\ell_1}(A) \hat{\otimes}_{\pi}  \lambda^{\ell_2}(B), \lambda^{\ell_3}(C)) \in \mathcal{BF}$ when $\lambda^{\ell_1}(A)$ and $\lambda^{\ell_2}(B)$ are nuclear.   
\end{abstract}
	
\dedicatory{Dedicated to the memory of Prof. Dr. Tosun Terzio{\~g}lu}

\section{Introduction}
     For an infinite matrix $A= (a_{n}^{k})$ with $0 \leq a_{n}^{k} \leq a_{n}^{k+1}$ and $\displaystyle \sup_{k} a_{n}^{k} > 0$ for every $n$ and $k$ we denote by $\lambda(A)$ the corresponding $l_{1}$-K\"{o}the space, that is,
		$$ 
		\lambda(A) = \left\{x = (x_{n}) : \left\|x\right\|_{k} = \sum_{n} \left|x_{n}\right| a_{n}^{k} < \infty, \quad \forall k \in \mathbb{N}\right\}
		$$
		Equipped with the system of seminorms $\left\{\left\|.\right\|_k, k \in \mathbb{N}\right\}$, $\lambda(A)$ is a Fr{\'e}chet space.
		
		Following \cite{Dra83}, we denote by $\ell$ a Banach sequence space in which the canonical system $(e_{n})$ is an unconditional basis. The norm $\left\|.\right\|$ is called monotone if $\left\|x\right\| \leq \left\|y\right\|$ whenever $\left|x_{n}\right| \leq \left|y_{n}\right|$, $x=(x_{n})$, $y=(y_{n}) \in \ell$, $n \in \mathbb{N}$. Let $\Lambda$ be the class of such spaces with monotone norm. In particular, $l_{p} \in \Lambda$ and $c_{0} \in \Lambda$. It is known that every Banach space with an unconditional basis has a monotone norm which is equivalent to its original norm. Indeed, it is enough to put
		$$ 
		\left\|x\right\| = \sup_{\left|\beta_{n}\right| \leq 1} \left|\sum_{n} e_{n}{'}({x}) \beta_{n} e_{n}\right|
		$$
		where $\left|.\right|$ denotes the original norm, $(e_{n}{'})$ denote the sequence of coefficient functionals.
		
		Let $\ell \in \Lambda$ and $\left\|.\right\|$ be a monotone norm in $\ell$. If $A = (a_{n}^{k})$ is a K\"{o}the matrix, the $\ell$-K\"{o}the space $\lambda^{\ell}(A)$ is the space of all sequences of scalars $(x_{n})$ such that $(x_{n}a_{n}^{k}) \in \ell$ with the topology generated by the seminorms
		$$
		\left\|(x_{n})\right\|_{k} = \left\|(x_{n} a_{n}^{k})\right\|
		$$
		Let us remind that $\left\|(e_{n})\right\|_{k} = a_{n}^{k}.$
    We denote by $\mathcal{L}(E,F)$ the space of all continuous linear operators between Fr{\'e}chet spaces $E$ and $F.$ For $T \in \mathcal{L}(E,F)$ we consider the following operator seminorms
		$$
		\left|T\right|_{p,q} = \sup \left\{ \left|Tx\right|_{p} : \left|x\right|_{q} \leq 1\right\}, \quad p,q \in \mathbb{N}
		$$
		which may take the value $+\infty$. In particular, for any one dimensional operator $ T = u \otimes x$ which sends each $z \in E$ to $u(z)x \in F$, we have
		$$                                       
		\left|T\right|_{p,q} = \left|u\right|_{q}^{*} \left|x\right|_{p}
		$$
where ${\left\|u\right\|}_{q}^{*} = \sup \left\{ \left|u(x)\right| : \left\|x\right\|_{q} \leq 1 \right\}$.

	  We recall that $T : E \longrightarrow F$ is continuous if there is a map $N: \mathbb{N} \longrightarrow \mathbb{N}$ such that
$$
{\left\| T \right\|}_{k,N(k)} < \infty, \quad \forall k \in \mathbb{N};
$$	
$T$ is bounded if $\exists N \in \mathbb{N}$ such that 
$$
		\displaystyle {\left\| T \right\|}_{r,N} < \infty, \quad \forall r \in \mathbb{N}.
$$
    
		We write $(E,F) \in \mathcal{B}$ if every continuous linear operator from $E$ to $F$ is bounded. For Fr{\'e}chet spaces $E$ and $F$, in \cite{Vog83}, Vogt proved that $(E,F) \in \mathcal{B}$ if and only if for every sequence $N(k)$, $\exists N \in \mathbb{N}$ such that $\forall r \in \mathbb{N}$ we have $k_{0} \in \mathbb{N}$ and $C>0$ with
\begin{eqnarray}
\displaystyle {\left\| T \right\|}_{r,N} \leq C \max_{1 \leq k \leq k_{0}} {\left\| T \right\|}_{k,N(k)}
\label{eqn}
\end{eqnarray}
for all $T \in \mathcal{L}(E,F)$.

    We say that a triple $(E,F,G)$ has the bounded factorization property and write $(E,F,G) \in \mathcal{BF}$ if each continuous linear operator $T : E \longrightarrow G$ that factors over $F$ is bounded. In \cite{Ter03}, the property $\mathcal{BF}$ is characterized not only for triples of K\"{o}the spaces but also for the general case of Fr{\'e}chet spaces. Our aim here to extend some results in \cite{Ter03} to $\ell$-K\"{o}the space case. 
		
\section{BOUNDED OPERATORS TO $\ell$-KOTHE SPACES}
	
    If we follow the steps of Crone and Robinson Theorem \cite{Cro75}, we obtain the following.\\
\begin{lemma} \label{lem} $T \in \mathcal{L}(\lambda(A),{\lambda}^{\ell}(B))$ iff $\forall m$, $\exists k$ such that
$$
\sup_{n} \frac{{\left\|Te_{n}\right\|}_{m}}{{\left\|e_{n}\right\|}_{k}} < +\infty
$$
\end{lemma}
\begin{proof}
$T \in \mathcal{L}(\lambda(A), \lambda^{\ell}(B))$ iff $\forall m$, $\exists k$ such that 
$$\sup_{x \neq 0, x \in \lambda(A)} \frac{{\left\|Tx\right\|}_{m}}{{\left\|x\right\|}_{k}} < +\infty$$ 
For $x=e_{n}$, we obtain the result.\\
Conversely, suppose that $\forall m$, $\exists k$ such that 
$$\sup_{n} \frac{{\left\|Te_{n}\right\|}_{m}}{{\left\|e_{n}\right\|}_{k}} < +\infty$$
Let $x \in \lambda(A)$.\\
\begin{eqnarray*}
{\left\|Tx\right\|}_{m} & = & { \left\|\sum_{n} x_{n} Te_{n}\right\|}_{m} \leq \sum_{n} \left|x_{n}\right| \displaystyle \frac{{\left\|Te_{n}\right\|}_{m}}{{\left\|e_{n}\right\|}_{k}} {\left\|e_{n}\right\|}_{k} \\
& \leq & \sup_{n} \frac{{\left\|Te_{n}\right\|}_{m}}{{\left\|e_{n}\right\|}_{k}} \displaystyle \sum_{n} \left|x_{n}\right| a_{n}^{k} \leq \sup_{n} \frac{{\left\|Te_{n}\right\|}_{m}}{{\left\|e_{n}\right\|}_{k}} {\left\|x\right\|}_{k}
\end{eqnarray*}
So, $T \in \mathcal{L}(\lambda(A), \lambda^{\ell}(B))$.
\end{proof}
Notice that when domain is $\ell$-K\"{o}the space, we can not use this argument.

     Our first result is the following.
\begin{theorem}\label{thm1} The following are equivalent:\\
i) $(\lambda(A),{\lambda}^{\ell}(B)) \in \mathcal{B}$\\
ii) for every sequence $N(k)$, there is $N \in \mathbb{N}$ such that for each $r \in \mathbb{N}$ we have $k_{0} \in \mathbb{N}$ and $C > 0$ with
$$\frac{b_{v}^{r}}{a_{i}^{N}} \leq C \max_{1 \leq k \leq k_{0}} \frac{b_{v}^{k}}{a_{i}^{N(k)}}$$
for all $v \in \mathbb{N}$,$i \in \mathbb{N}.$
\end{theorem}
\begin{proof} 
Suppose that $(\lambda(A),{\lambda}^{\ell}(B)) \in \mathcal{B}$. Consider $S: \lambda(A) \longrightarrow \lambda^{\ell}(B)$ with $S = e_{i}{'} \otimes e_{v}$ where $e_{i}{'}(x) = x_{i}$ for all $x \in \lambda(A).$ Since $S$ is the operator of rank one, we note that
$$
{\left\|S\right\|}_{k,N(k)} = {\|e_{i}^{'}\|}_{N(k)} {\|e_{v}\|}_{k} = \frac{b_{v}^{k}}{a_{i}^{N(k)}}
$$
Similarly, ${\left\|A\right\|}_{r,N} = \frac{b_{v}^{r}}{a_{i}^{N}}$. So, the result follows from \eqref{eqn}.
    
		For the converse, let $\displaystyle Te_{i} = \sum_{v=1}^{\infty} u_{vi}e_{v}$. Since $T$ is continuous, by Lemma \ref{lem}, there is $N(k)$ such that
\begin{eqnarray*}
\left\|T\right\|_{k,N(k)} & = & \sup_{i \in \mathbb{N}} \frac{{\left\|Te_{i}\right\|}_{k}}{{\left\|e_{i}\right\|}_{N(k)}} \\
 & = &\sup_{i \in \mathbb{N}} \sup_{\left| \beta_{v} \right| \leq 1} \left|\sum_{v=1}^{\infty} u_{vi} \beta_{v} \frac{b_{v}^{k}}{a_{i}^{N(k)}}e_{v} \right| < \infty 
\end{eqnarray*}
     
		So we find $N \in \mathbb{N}$ such that
\begin{eqnarray*}
\left\|T\right\|_{r,N} & \leq & \sup_{i \in \mathbb{N}} \left\{\sup_{\left| \beta_{v} \right| \leq 1} \left| \sum_{v=1}^{\infty} u_{vi} \beta_{v} \frac{b_{v}^{r}}{a_{i}^{N}} e_{v} \right|\right\} \\
 & \leq & \sup_{i \in \mathbb{N}} \left\{\sup_{\left| \beta_{v} \right| \leq 1} \left|\sum_{v=1}^{\infty} u_{vi} \beta_{v} \left( C \max_{1 \leq k \leq k_{0}} \frac{b_{v}^{k}}{a_{i}^{N(k)}} \right) e_{v} \right|\right\} \\
 & \leq & C \sum_{k=1}^{k_{0}} \sup_{i \in \mathbb{N}} \left\{\sup_{\left| \beta_{v} \right| \leq 1} \left|\sum_{v=1}^{\infty} u_{vi} \beta_{v} \frac{b_{v}^{k}}{a_{i}^{N(k)}} e_{v} \right|\right\}< \infty
\end{eqnarray*}
Therefore, $T$ is bounded.
\end{proof}
Now, consider the $\ell$-K\"{o}the space $\lambda^{\ell}(A)$ and any Fr{\'e}chet space $E$. Then, we obtain the following.
\begin{theorem}\label{thm2} The following are equivalent:\\
i) $(E,\lambda^{\ell}(A)) \in \mathcal{B}$\\
ii) for every sequence $N(k)$, there is $N \in \mathbb{N}$ such that for each $r \in \mathbb{N}$ we have $k_{0} \in \mathbb{N}$ and $C>0$ with
$$a_{v}^{r} {\left\|u\right\|}_{N}^{*} \leq C \max_{1 \leq k \leq k_{0}} a_{v}^{k} \left\|u\right\|_{N(k)}^{*}$$
for all $v \in \mathbb{N}$, $u \in E{'}$.
\end{theorem}
\begin{proof} 
Suppose that $(E,\lambda^{\ell}(A)) \in \mathcal{B}$. Similar to the proof of Theorem \ref{thm1}, consider the operator of rank one $A=y \otimes e_{v}$ where $y \in E{'}$. The result follows from \eqref{eqn}.

     For the converse, let $T: E \longrightarrow \lambda^{\ell}(A)$ be continuous linear operator. Let
$$
Tx = \sum_{v=1}^{\infty} e_{v}{'}(Tx) e_{v} = (e_{v}{'}T(x)) = (u_{v}(x)), \quad x \in E
$$
where $ u_{v} = e_{v}{'} \circ T.$

     Then, by continuity we find $N(k)$ such that
$$
\sup_{{\left\|x\right\|}_{N(k)}\leq1} \left( \sup_{\left|\beta_{v}\right| \leq 1 } \left| \sum_{v} \beta_{v} {u_{v}(x)} a_{v}^{k} e_{v} \right|\right) = M(k)
$$
      
			Let $\left|u_{v}(x)\right| = \theta_{v} u_{v}(x)$ where $\theta_{v} = \pm 1$ and $\beta_{v} \theta_{v} = \alpha_{v}$, note that 
\begin{eqnarray*}
{\left\|Tx\right\|}_{r} & = & \sup_{\left|\beta_{v}\right| \leq 1} \left| \sum_{v} \beta_{v} u_{v}(x) a_{v}^{r} e_{v} \right| \\
  & \leq &  \sup_{\left| \beta_{v} \right| \leq 1 } \left| \sum_{v} \beta_{v} \left(\frac{\left|u_{v}(x)\right|}{\left\|x\right\|_{N}}\right) a_{v}^{r} e_{v} \right| {\left\|x\right\|_{N}} \\
  & \leq &  \sup_{\left| \beta_{v} \right| \leq 1 } \left| \sum_{v} \beta_{v} {\left\|u_{v}\right\|}_{N}^{*} a_{v}^{r} e_{v} \right|{\left\|x\right\|_{N}} \\
	& \leq &  \sup_{\left| \beta_{v} \right| \leq 1 } \left| \sum_{v} \beta_{v} \left( C \max_{1 \leq k \leq k_{0}} a_{v}^{k} \left\| u_{v}\right\|_{N(k)}^{*} \right) e_{v} \right|{\left\|x\right\|_{N}} \\
  & \leq &  C \sum_{k=1}^{k_{0}} \sup_{\left| \beta_{v} \right| \leq 1} \left| \sum_{v} \beta_{v} a_{v}^{k} \left(\sup_{{\left\|x\right\|}_{N(k)}\leq1} \left| u_{v}(x) \right|\right) e_{v} \right|{\left\|x\right\|_{N}} \\
	& \leq & C \sum_{k=1}^{k_{0}} \sup_{{\left\|x\right\|}_{N(k)}\leq1} \left(\sup_{\left| \beta_{v} \right| \leq 1} \left| \sum_{v} \beta_{v} a_{v}^{k} \theta_{v} u_{v}(x) e_{v}\right|\right){\left\|x\right\|_{N}} \\
	& \leq & C \sum_{k=1}^{k_{0}} \sup_{{\left\|x\right\|}_{N(k)}\leq1} \left(\sup_{\left| \alpha_{v} \right| \leq 1} \left| \sum_{v} \alpha_{v} a_{v}^{k} u_{v}(x) e_{v}\right|\right){\left\|x\right\|_{N}} \\
	& \leq & \left(C \sum_{k=1}^{k_{0}} M(k)\right) {\left\|x\right\|_{N}}
\end{eqnarray*}

    Hence $T$ is bounded.
\end{proof}

\section{BOUNDED FACTORIZATION PROPERTY FOR $\ell$-KOTHE SPACES}
    We need the following theorem \cite[Theorem 2.2]{Ter03}.
\begin{theorem}\label{thm3}
For Fr{\'e}chet spaces $E, F$ and $G$ we have $(E,F,G) \in \mathcal{BF}$ if and only if for every sequence $N(k)$ there is $N \in \mathbb{N}$ such that for each $r \in \mathbb{N}$ we have $k_{0} = k_{0}(r) \in \mathbb{N}$ and $C = C(r)> 0$ so that the following inequality
$$
\left\|T\right\|_{r,N} \leq C \max_{1 \leq k \leq k_{0}} \left\{\left\|R\right\|_{k,N(k)}\right\} \max_{1 \leq k \leq k_{0}}\left\{\left\|S\right\|_{k,N(k)}\right\}
$$
is satisfied for every $R \in L(F,G)$, $S \in L(E,F)$ where $T = RS$.
\end{theorem}

     The next result is obtained by following the lines of \cite[Corollary 3.1]{Ter03}.
\begin{theorem}\label{thm4}
Let $E$ be a Fr{\'e}chet space, $\lambda^{\ell}(B),\lambda^{\tilde{\ell}}(C)$ be ${\ell}$-K\"{o}the spaces and $\lambda^{\ell}(B)$ be nuclear. Then $(E,\lambda^{\ell}(B),\lambda^{\tilde{\ell}}(C)) \in \mathcal{BF}$ if and only if for every sequence $N(k)$ there is $N \in \mathbb{N}$ such that for each $r \in \mathbb{N}$ we have $k_{0} \in \mathbb{N}$ and $C > 0$ with
\begin{equation}\label{eqn1}
c_{j}^{r} \left\|u\right\|_{N}^{*} \leq C \max_{1 \leq k \leq k_{0}} \left\{\left\|u\right\|_{N(k)}^{*} b_{i}^{k}\right\} \max_{1 \leq k \leq k_{0}} \left\{\frac{c_{j}^{k}}{b_{i}^{N(k)}}\right\}
\end{equation}
for all $i \in \mathbb{N}, j \in \mathbb{N}$ and $u \in E{'}.$
\end{theorem}
\begin{proof}
Let $S = u \otimes e_{i}$ and $R = e_{i}{'} \otimes e_{j}$ where $u \in E{'}$. Then $RS: E \longrightarrow G $ is the operator of rank one which sends each $x \in E$ to $u(x)e_{j}$. If we apply Theorem \ref{thm3} we obtain the result.

     For sufficiency, we take $S \in L(E,\lambda^{\ell}(B))$, $R \in L(\lambda^{\ell}(B), \lambda^{\tilde{\ell}}(C))$ and $T=RS$. Since $\lambda^{\ell}(B)$ is nuclear, $\exists S(k)$ such that $S(k) > N(k)$ and
$$
\sum_{i} \frac{b_{i}^{N(k)}}{b_{i}^{S(k)}} = \theta(k) < \infty, \quad k \in \mathbb{N}.
$$
We can write $\displaystyle Sx = \sum_{i} u_{i}(x) e_{i}$ where $u_{i} = e_{i}{'} \circ S \in E{'}$ and $\displaystyle Re_{i} = \sum_{j} r_{ji}e_{j}.$ Therefore
$$ 
Tx = \sum_{i} \sum_{j} u_{i}(x) r_{ji} e_{j}
$$
 
     For this $S(k)$ we choose $N \in \mathbb{N}$ such that for each $r \in \mathbb{N}$ we obtain $k \in \mathbb{N}$ and $C > 0$ with
\begin{equation}\label{eqn2}
c_{j}^{r} \left\|u_{i}\right\|_{N}^{*} \leq C \max_{1 \leq k \leq k_{0}} \left\{\left\|u_{i}\right\|_{S(k)}^{*} b_{i}^{k}\right\} \max_{1 \leq k \leq k_{0}} \left\{\frac{c_{j}^{k}}{b_{i}^{S(k)}}\right\}
\end{equation}
for all $i \in \mathbb{N}, j \in \mathbb{N},u_{i} \in E{'}.$

    Since all types nuclear K\"{o}the spaces determined by one and the same matrix $B$ coincide \cite[Corollary 2, p.22]{Dra83}$, \lambda(B) = \lambda^{\ell}(B)$ and we have
$$
\left\|S\right\|_{k,N(k)} = \sup_{\left\|x\right\|_{N(k)} \leq 1} \left\|Sx\right\|_{k} = \sup_{\left\|x\right\|_{N(k)} \leq 1} \sum_{i} \left|u_{i}(x)\right| b_{i}^{k} = \sum_{i} \left\|u_{i}\right\|_{N(k)}^{*} b_{i}^{k}
$$
$$
\left\|R\right\|_{k,N(k)} = \sup_{i} \frac{\left\|Re_{i}\right\|_{k}}{\left\|e_{i}\right\|_{N(k)}} = \sup_{i} \sup_{\left|\beta_{j}\right|\leq1} \left|\sum_{j} \frac{r_{ji} \beta_{j} c_{j}^k e_{j}}{b_{i}^{N(k)}} \right|
$$
Therefore, we have
$$
\sum_{i} \left\|u_{i}\right\|_{S(k)}^{*} b_{i}^{k} \leq \sum_{i} \left\|u_{i}\right\|_{N(k)}^{*} b_{i}^{k} = \left\|S\right\|_{k,N(k)}
$$
and
$$
\sum_{j} \frac{\left|r_{ji}\right| c_{j}^{k}}{b_{i}^{S(k)}} \leq \theta(k) \sup_{j} \frac{\left|r_{ji}\right| c_{j}^{k}}{b_{i}^{N(k)}} \leq \theta(k) \left\|R\right\|_{k,N(k)} 
$$
(see \cite[proof of Corollary 2, p.22]{Dra83})
Hence, by \eqref{eqn2} we obtain that
\begin{align*}\label{eqn3}
\left\|Tx\right\|_{r} & \leq \sum_{i} \sum_{j} \frac{\left|u_{i}(x)\right|}{\left\|x\right\|_{N}} \left|r_{ji}\right| c_{j}^{r} \left\|x\right\|_{N}\\ 
                      & \leq \sum_{i} \sum_{j} \left\|u_{i}\right\|_{N}^{*} \left|r_{ji}\right| c_{j}^{r} \left\|x\right\|_{N}\\
											& \leq \sum_{i} \sum_{j} \left|r_{ji}\right| \left\{ C \max_{1 \leq k \leq k_{0}} \left\{ \left\|u_{i}\right\|_{S(k)}^{*} b_{i}^{k} \right\} \max_{1 \leq k \leq k_{0}} \left\{ \frac{c_{j}^{k}}{b_{i}^{S(k)}} \right\} \right\} \left\|x\right\|_{N}\\
											& \leq C \sum_{k=1}^{k_{0}} \sum_{i} \left\|u_{i}\right\|_{S(k)}^{*} b_{i}^{k} \sum_{j} \frac{\left|r_{ji}\right|c_{j}^{k}}{b_{i}^{S(k)}} \left\|x\right\|_{N}\\
											& \leq C \sum_{k=1}^{k_{0}} \sum_{i} \left\|u_{i}\right\|_{S(k)}^{*} b_{i}^{k} \theta(k) \left\|R\right\|_{k,N(k)} \left\|x\right\|_{N}\\
											& \leq \left( C \sum_{k=1}^{k_{0}} \theta(k) \left\|S\right\|_{k,N(k)} \left\|R\right\|_{k,N(k)} \right) \left\|x\right\|_{N} 
\end{align*}
Therefore, $T$ is bounded.
\end{proof}

      Recall that projective tensor product of two $l_{1}$-K\"{o}the spaces $\lambda(A)$ and $\lambda(B)$ is isomorphic to $\lambda(D)$ where $d_{vz}^{k} = a_{v}^{k} b_{z}^{k}.$
			
			Theorem \ref{thm4} enables us to get:
\begin{theorem}
Suppose $(E, \lambda^{\ell_1}(A)) \in \mathcal{B}$ and $(\lambda^{\ell_2}(B), \lambda^{\ell_3}(C)) \in \mathcal{B}$ where $E$ is a Fr{\'e}chet space, $\lambda^{\ell_3}(C)$ is an $\ell$-K\"{o}the space, $\lambda^{\ell_1}(A)$ and $\lambda^{\ell_2}(B)$ are nuclear $\ell$-K\"{o}the spaces. Then $(E,\lambda^{\ell_1}(A) \hat{\otimes}_{\pi} \lambda^{\ell_2}(B), \lambda^{\ell_3}(C)) \in \mathcal{BF}$.
\end{theorem}
\begin{proof} Given $N(k)$ which is assumed to be non-decreasing.  Since $\lambda^{\ell_2}(B)$ is nuclear and $(\lambda^{\ell_2}(B),\lambda^{\ell_3}(C))\in\mathcal{B}$, we obtain that $\lambda^{\ell_2}(B) = \lambda(B)$ and by Theorem \ref{thm1}, $\exists n \in \mathbb{N}$ such that $\forall r \in \mathbb{N}$ we have $k_{0} = k_{0}(r) \in \mathbb{N}$ and $C_{1} = C_{1}(r) > 0$ with
$$
\frac{c_{j}^{r}}{b_{i}^{n}} \leq C_{1} \max_{1 \leq k \leq k_{0}} \frac{c_{j}^{k}}{b_{i}^{N(k)}}
$$ 
for all $i \in \mathbb{N}$, $j \in \mathbb{N}.$

       We then determine $S(k)$ such that $S(k) = N(n)$ if $1 \leq k \leq n$ and $S(k) > N(k)$ if $n+1 \leq k \leq s_{0}$. Since $(E, \lambda^{\ell_1}(A)) \in \mathcal{B}$ by Theorem \ref{thm2}, for this $S(k)$, we find $m \in \mathbb{N}$ such that $\forall q \in \mathbb{N}$ we have $s_{0} = s_{0}(q)$ and $C_{2} = C_{2}(q)$
\begin{align*}
			a_{v}^{q} \left\|u\right\|_{m}^{*} & \leq C_{2} \max_{1 \leq k \leq s_{0}} a_{v}^{k} \left\|u\right\|_{S(k)}^{*}\\
			                                   & \leq C_{2} \max_{n \leq k \leq s_{0}} a_{v}^{k} \left\|u\right\|_{N(k)}^{*}
\end{align*}

      Therefore, for this $N(k)$ we have $\tilde{s_{0}} = s_{0}(N(k))$ and $\tilde{C_{2}} = C_{2}(N(k))$ with
\begin{align*}
c_{j}^{r} \left\|u\right\|_{m}^{*} & = \frac{c_{j}^{r}}{b_{i}^{n}} b_{i}^{n} \left\|u\right\|_{m}^{*} \\
                                   & \leq \left\{C_{1} \max_{1 \leq k \leq k_{0}} \frac{c_{j}^{k}}{b_{i}^{N(k)}}\right\} b_{i}^{n} \left\|u\right\|_{m}^{*} \\
																	 & \leq \left\{C_{1} \max_{1 \leq k \leq k_{0}} \frac{c_{j}^{k}}{a_{v}^{N(k)} b_{i}^{N(k)}}\right\} a_{v}^{N(k)} b_{i}^{n} \left\|u\right\|_{m}^{*} \\
																	 & \leq \left\{C_{1} \max_{1 \leq k \leq k_{0}} \frac{c_{j}^{k}}{a_{v}^{N(k)} b_{i}^{N(k)}}\right\} \left\{ \tilde{C_{2}} \max_{n \leq k \leq \tilde{s_{0}}} a_{v}^{k} \left\|u\right\|_{N(k)}^{*} b_{i}^{n} \right\}\\
																	 & \leq \left\{C_{1} \max_{1 \leq k \leq k_{0}} \frac{c_{j}^{k}}{a_{v}^{N(k)} b_{i}^{N(k)}}\right\} \left\{ \tilde{C_{2}} \max_{1 \leq k \leq \tilde{s_{0}}} a_{v}^{k} b_{i}^{k} \left\|u\right\|_{N(k)}^{*} \right\}
\end{align*}

      Let $s = \max\left\{k_{0},\tilde{s_{0}}\right\}$ and $C = C(r) = C_{1}\tilde{C_{2}}$. We have proved that $\exists m \in \mathbb{N}$ such that $\forall r \in  \mathbb{N}$ we have $s \in \mathbb{N}$ and $C > 0$ with 
$$
c_{j}^{r} \left\|u\right\|_{m}^{*} \leq C \max_{1 \leq k \leq s} \left\{ \frac{c_{j}^{k}}{a_{v}^{N(k)}b_{i}^{N(k)}} \right\} \max_{1 \leq k \leq s} \left\{ \left\|u\right\|_{N(k)}^{*} a_{v}^{k} b_{i}^{k} \right\}
$$
for all $j \in \mathbb{N}, v \in \mathbb{N}, i \in \mathbb{N}$ and $u \in E{'}.$

If $\lambda(A)^{\ell_1}$ and $\lambda^{\ell_2}(B)$ are nuclear $\ell$-K\"{o}the spaces, then 
$$\lambda^{\ell_1}(A) \hat{\otimes}_{\pi} \lambda^{\ell_2}(B) \cong \lambda(A) \hat{\otimes}_{\pi} \lambda(B) \cong \lambda(D)$$ is nuclear where $d_{vi}^{k} = a_{v}^{k} b_{i}^{k}$ \cite[Corollary 2, p.22]{Dra83}.

      By Theorem \ref{thm4} we obtain that $(E,\lambda^{\ell_1}(A) \hat{\otimes}_{\pi} \lambda^{\ell_2}(B), \lambda^{\ell_3}(C)) \in \mathcal{BF}$.
\end{proof}

\end{document}